\documentclass{amsart}
\usepackage{amsfonts,latexsym,graphicx,amssymb,amsmath,url,amsthm}
\newtheorem{theorem}{Theorem}
\newtheorem{lemma}{Lemma}
\newtheorem{corollary}{Corollary}

\theoremstyle{definition}
\theoremstyle{remark}
\newtheorem*{remark}{{\bf{Remark}}}
\newcommand{\N}{\mathbb N}
\newcommand{\Z}{\mathbb Z}

\begin{document}

\title{Hilbert cubes in
 progression-free sets and in the set of squares II}

\author{Rainer Dietmann}
\address{Department of Mathematics,
Royal Holloway, University of London, Egham, TW20 0EX Surrey, UK}
\email{Rainer.Dietmann@rhul.ac.uk}
\author{Christian Elsholtz}
\address{Institut f\"ur Analysis und Computational Number Theory,
Technische Universit\"at Graz,
Steyrergasse 30,
A-8010 Graz, Austria}
\email{elsholtz@math.tugraz.at}

\date{\today}
\begin{abstract}
Let $S_2$ be the set of integer squares.
We show that the dimension $d$ of a Hilbert cube 
$a_0+\{0,a_1\}+\cdots + \{0, a_d\}\subset S_2$ 
is bounded by $d=O(\log \log N)$.
\end{abstract}
\maketitle
Let
\begin{align*}
   H(a_0;a_1, \ldots , a_d) & :=a_0+\{0,a_1\}+\cdots + \{0, a_d\} \\
   & = \left\{ a_0 + \sum_{i=1}^d \epsilon_i a_i : \epsilon_i \in \{0,1\}
   \right\}
\end{align*}
denote 
a Hilbert cube of dimension $d$.
Brown, Erd\H{o}s and Freedman \cite{BrownandErdosandFreedman:1990}
asked whether the maximal
dimension of a Hilbert cube in the set of squares 
is absolutely bounded or not. Experimentally, one finds only very small cubes
such as 
\[
   1+\{0,840\}+\{0,840\}+\{0,528\}
   = \{1^2, 23^2, 29^2, 37^2, 29^2, 37^2, 41^2, 47^2\}.
\]
Observe that in this example
$29^2$ and $37^2$ occur as sums in two different ways.
Cilleruelo and Granville \cite{CillerueloandGranville:2007} and 
Solymosi \cite{Solymosi:2007} and Alon, Angel, Benjamini and Lubetzky
\cite{AlonandAngelandBenjaminiandLubetzky:2012}
explain that the Bombieri-Lang conjecture 
implies that $d$ is absolutely bounded.
Hegyv\'ari and S\'ark\"ozy
(\cite{HegyvariandSarkozy:1999}, Theorem 1) proved that for the set of 
integer squares
$S_2\cap [1, N]$ the maximal dimension is bounded by $d=O((\log N)^{1/3})$.
Diet\-mann and Elsholtz (\cite{DietmannandElsholtz:2012}, Theorem 3)
improved this to $d=O((\log \log N)^2)$. Here we further reduce that bound.

\begin{theorem}[Main theorem]{\label{thm:squares}}
Let $S_2$ denote the set of integer squares.
Let $N$ be sufficiently large, let 
$a_0$ be a non-negative integer and
let $A=\{a_1, \ldots, a_d\}$ be a set of distinct positive integers
such that 
$H(a_0; a_1, \ldots , a_d)\subseteq S_2\cap [1,N]$. Then
\[
  d \leq 7 \log \log N.
\]
\end{theorem}
A comparable bound was proved in the author's earlier paper
(\cite{DietmannandElsholtz:2012}, Theorem 1)
for higher powers instead of squares.
\begin{remark}
The special case of subsetsums, i.e.~Hilbert cubes with
$a_0=0$, was previously  
studied by Csikv\'{a}ri (\cite{Csikvari:2008}, Corollary 2.5), who proved
in this case the same bound $d =O( \log \log N)$. His method of proof 
would not extend to the general case of $a_0 \neq 0$. 
\end{remark}
\begin{corollary}{\label{cor:quadraticpoly}}
Let $f(x)=ax^2+bx+c$ be a quadratic polynomial where $a,b,c \in \Z$
such that $a>0$, and let $S=\{f(x) : x \in \N\}$.
Let $a_0$ be a non-negative integer and $A=\{a_1, \ldots, a_d\}$ be a
set of distinct positive integers such that $H(a_0; a_1, \ldots, a_d) \subset
S \cap [1,N]$. Then for sufficiently large $N$, we have
\[
  d \le 7 \log \log N.
\]
\end{corollary}

\begin{theorem}{\label{thm:kprogression}}
Let $k\geq 3$ be a positive integer, and
let $S$ denote a set of integers without an arithmetic progression of 
length $k$. Moreover, let $c$ be a real number such that
\[
  1<c<\frac{k}{k-1}.
\]
Then for sufficiently large $N$, the following holds true: If 
$a_0$ is a non-negative integer and
$A=\{a_1, \ldots, a_d\}$ is a set of distinct positive integers
such that 
$H(a_0; a_1, \ldots , a_d)$ $\subseteq S\cap [1,N]$, then
\[ d\leq \frac{2(k-2)}{(k-1)\log c}\log N.\]
\end{theorem}
This last theorem can be easily deduced from a recent paper of Schoen 
\cite{Schoen:2011}, but we state it for completeness. The two following lemmas
are also implicitly contained in Schoen \cite{Schoen:2011}, Lemma 2.1.

\begin{lemma}\label{growth}
Let $k\geq 3$ be a positive integer, and
let $S$ denote a set of integers without an arithmetic progression of 
length $k$. Moreover, let $c$ be a real number such that
\[
  1<c<\frac{k}{k-1}.
\]
Let $H=a_0+\{0,a_1\} + \cdots + \{0,a_d\}\subset S$. 
Then $|H|\geq 2c^{d-1}$.
\end{lemma}
It is well known that for sets
without progressions of length $k=3$ one even has that $|H|=2^d$. 
This is an exercise in Solymosi \cite{Solymosi:2007}, see also Lemma 3
in \cite{DietmannandElsholtz:2012}.

\begin{lemma}\label{dense-progression}
Let $0<\alpha<1$, let $h$ be an integer and let $B$ be a non-empty
set of distinct
integers. If $|B \cap (B+h)|> (1-\alpha) |B|$, 
then $B$ contains an arithmetic progression
of length $\lfloor \frac{1}{\alpha} \rfloor +1$ and difference $h$.
\end{lemma}
\begin{proof}[Proof of Lemma \ref{dense-progression}]
Consider the shift operator $f:\Z \rightarrow \Z$ defined by $f(b)=b+h$
and its iterates. For given $b \in \Z$,
let $r(b)$ denote the least non-negative integer $r$ with 
\[
\{b, f(b), \ldots ,f^{r-1}(b)\} \subset  B, \text{ but }f^r(b)=b+rh\not\in B.\]
The  assumption $|\{b \in B:  b+h \in B\}| > (1-\alpha)|B|$
implies that for each fixed non-negative integer $r$, 
there are less than $\alpha |B|$ elements $b \in B$ with
this given value $r=r(b)$. Hence, for $k \in \N$
the number of elements $b\in B$
with $r(b) \leq k$ is less than 
$k \alpha |B| $.
If $k \alpha |B|<|B|$, then
there exists a $b \in B$ with 
$r(b)\geq k+1 \geq \lfloor \frac{1}{\alpha}\rfloor +1$. 
Therefore $B$ contains
the arithmetic progression $\{b, b+h ,\ldots, b+(r-1)h\}$
of length $r\geq \lfloor \frac{1}{\alpha}\rfloor +1$. 
\end{proof}
\begin{proof}[Proof of Lemma \ref{growth}]
Let
\[
  H_i=a_0+\{0,a_1\} + \cdots + \{0,a_i\}.
\]
Suppose that
$|H|< 2c^{d-1}$, then there is some $i \in \{1, \ldots , d-1\}$ 
such that
\[
  \frac{|H_{i+1}|}{|H_i|}<c.
\]
For this $i$ we have
\[
  |(H_i+a_{i+1})\cap H_i|> (2-c)|H_i|.
\]
Then by Lemma \ref{dense-progression} and our assumption on $c$ the set 
$H_i$ contains an arithmetic progression of length 
\[
  \lfloor \frac{1}{c-1}\rfloor +1 \ge k
\]
which is a contradiction, as $S$ does not contain a progression of length $k$.
\end{proof}

As in our previous work
\cite{DietmannandElsholtz:2012}
 we make use of the following two results on squares:

\begin{lemma}[Theorem 9 of Gyarmati \cite{Gyarmati:2001}]
\label{kati}
Let $S_2$ denote the set of integer squares.
For sufficiently large $N$ the following holds true: 
If $C, D \subseteq \{1, \ldots, N\}$ such that $C+D\subseteq S_2$, then
\[
  \min\{|C|, |D|\} \leq 8 \log N.
\]
\end{lemma}
\begin{lemma} [Fermat, Euler
(see Volume II, page 440 of \cite{Dickson:1966})]{\label{Fermat}}
There are no four integer squares in arithmetic progression.
\end{lemma}
\begin{corollary}{\label{cor:fermatshift}}
Let $f(x)=ax^2+bx+c$ be a quadratic polynomial where $a,b,c \in \Z$
such that $a>0$, and let $S=\{f(x) : x \in \N\}$.
Then the set $S$ does not contain
 four integer squares in arithmetic progression.
\end{corollary}

Let us remark that 
Setzer \cite{Setzer:1979} proved this corollary using elliptic curves.
Here we show that it is a simple consequence of Lemma
\ref{Fermat}, which can be proved with elliptic curves, but also
allows for an elementary proof.
\begin{proof}[Proof of Corollary \ref{cor:fermatshift}]
From $4af(x)+b^2-4ac=(2ax+b)^2$,
it follows that if the arithmetic four-progression
$P=\{n,n+m, n+2m, n+3m\}$ is contained in $S$, 
then the shifted four-progression
$\{4an+b^2-4ac, 4an+b^2-4ac+4am,4an+b^2-4ac+8am,4an+b^2-4ac+12am\}$
 will be in the set of squares, contradicting Lemma \ref{Fermat}.
\end{proof}
\begin{proof}[Proof of Theorem \ref{thm:squares}]
By Lemma \ref{Fermat}, Lemma \ref{growth} is applicable for $S=S_2$ with any
$c<\frac{4}{3}$.
Let
\[
  C=H(a_0; a_1, \ldots , a_{\lfloor d/2\rfloor})
\]
and
\[
  D= H(0; a_{\lceil (d+1)/2\rceil}, \ldots , a_d),
\]
then by Lemma \ref{growth} and Lemma \ref{kati} we obtain
\[c^{\lfloor d/2\rfloor }\leq \min( |C|, |D|)\leq 8 \log N.\]
Hence
\[ d \leq \frac{2}{\log c} \log \log N+O(1)
\leq  6.96\log \log N,\]
for sufficiently large $N$.
\end{proof}
\begin{proof}[Proof of Theorem \ref{thm:kprogression}]
The proof is as the one above, except that
instead of Lemma \ref{kati} we use the bound
\[
  \min( |C|, |D|)\leq 3N^{1-\frac{1}{k-1}},
\]
by Croot, Ruzsa and Schoen \cite{CrootandRuzsaandSchoen:2007}.
\end{proof}
\begin{proof}[Proof of Corollary \ref{cor:quadraticpoly}]
From $4af(x)=(2ax+b)^2-b^2+4ac$,
it follows that if
\[
  H(a_0; a_1, \ldots, a_d) \subset S \cap [1,N],
\]
then
\begin{align*}
  4a H(a_0; a_1, \ldots, a_d) + b^2-4ac & \subset S_2 \cap
  [4a+b^2-4ac, 4aN+b^2-4ac] \\
  & \subset S_2 \cap [1, 4aN+b^2-4ac].
\end{align*}
Moreover,
\[
  4a H(a_0; a_1, \ldots, a_d) + b^2-4ac 
  = H(4a a_0+b^2-4ac; 4a a_1, \ldots, 4a a_d).
\]
The Corollary now follows immediately from the proof of Theorem 
\ref{thm:squares}, and observing that 
\[d < 6.96\log \log (4aN+b^2-4ac) + O(1) \leq 7 \log \log N\]
for sufficiently large $N$.
\end{proof}

\begin{remark}
It is easy to rephrase our results for multisets where multiplicity
of elements $a_i$ is allowed: If there are no $k$-progressions in $S$,
then it is immediately seen that the maximum multiplicity can be $k-1$,
so after allowing for an extra factor $k-1$ our bounds established above
still hold true. In particular, Theorem \ref{thm:squares}
holds with $d \leq 21 \log \log N$.
\end{remark}

\begin{remark}
While the proof follows the general strategy outlined in 
\cite{DietmannandElsholtz:2012}, the paper by Schoen \cite{Schoen:2011}
(which independently studied related questions) 
inspired us to rephrase his Lemma 2.1 in the form of our 
Lemma \ref{growth}. This allows us to simplify and improve 
some of the steps described in
\cite{DietmannandElsholtz:2012} considerably.

With regard to the earlier results in \cite{DietmannandElsholtz:2012},
Noga Alon kindly pointed out to us that
Lemma 5 of \cite{DietmannandElsholtz:2012} 
is actually a version of a result of Erd\H{o}s and Rado 
\cite{ErdosandRado:1960} on $\Delta$-systems (or sunflowers).
Small quantitative
improvements here are due to Kostochka \cite{Kostochka:1996}, 
which together with the previous 
argument, for the set $S_2$ of squares, would lead to the tiny improvement
$d=O((\log_2 N)^2\frac{\log_5 N}{\log_4 N})$, 
the $\log_i N$ denoting the $i$-fold iterated logarithm.
Moreover, the Erd\H{o}s-Rado
conjecture on these $\Delta$-systems,
for which Erd\H{o}s \cite{Erdos:1981} offered a prize of $\$$1000,
would have implied $d=O(\log \log N)$.
Fortunately, Lemma \ref{growth} allowed us to bypass the realm
of $\Delta$-systems.
\end{remark}
 
\textbf{Acknowledgements:} We would like to thank Noga Alon for
 pointing out to
us the connection to the work of Erd\H{o}s and Rado
\cite{ErdosandRado:1960}.
In an attempt to avoid their
conjecture we got the idea of the present approach.

During the preparation of this paper, 
R.~Dietmann was supported by EPSRC Grant EP/I018824/1,
C. Elsholtz by FWF-DK Discrete Mathematics Project W1230-N13.

\end{document}